\documentclass[12pt, reqno]{amsart} 

\usepackage{amsmath, amssymb, amsthm} 
\usepackage{geometry} 
\usepackage{enumitem} 
\usepackage{hyperref} 
\usepackage{xcolor}
\hypersetup{
    colorlinks=true,
    linkcolor=blue,
    citecolor=blue,
    urlcolor=blue
}
\usepackage[backend=biber, style=numeric, sorting=nyt]{biblatex} 
\allowdisplaybreaks

\theoremstyle{plain} 
\newtheorem{theorem}{Theorem}[section]
\newtheorem{lemma}[theorem]{Lemma}
\newtheorem{proposition}[theorem]{Proposition}

\theoremstyle{definition} 
\newtheorem{definition}[theorem]{Definition}

\theoremstyle{remark} 
\newtheorem{remark}[theorem]{Remark}

\begin{document}

\title{Conformal product structures on compact manifolds with constant sectional curvature}

\author{Xianfeng JIANG}
\address{Laboratoire mathématique d'Orsay, Université Paris-Saclay, 91405 Orsay, France}
\email{xianfeng.jiang@mail.ustc.edu.cn}



\begin{abstract}
We prove that compact non-flat manifolds of constant sectional curvature admit no conformal product structure. In the flat case, we show that in dimensions at least three every conformal product structure is trivial, namely its Weyl connection is the Levi-Civita connection of the given flat metric; flat surfaces are exceptional. Furthermore, we demonstrate that the methods extend naturally to irreducible, compact locally symmetric spaces of non-positive curvature.
\end{abstract}

\maketitle
\pagestyle{plain}

\section{Introduction}

On every Riemannian manifold $(M,g)$, there is a unique torsion-free linear connection $\nabla$ compatible with the Riemannian metric, called the Levi-Civita connection. More generally, on a conformal manifold $(M,c)$, where $c$ is a conformal class of Riemannian metrics, a torsion-free linear connection $D$ compatible with the conformal structure $c$ is called a Weyl connection. The concept of a Weyl connection was introduced in \cite{weylraumzeitmaterie} in order to unify Einstein's general theory of relativity (gravity) and Maxwell's electromagnetism within a purely geometric framework. 

Recall that by fixing a background metric $g$ in the conformal class $c$, there is a one-to-one correspondence between Weyl connections and 1-forms, i.e., for every Weyl connection $D$, there is a 1-form $\theta$ (called the Lee form), satisfying the formula $Dg=-2\theta\otimes g$, and conversely $D$ is uniquely determined by $\theta$. For a different choice of background metric $g'=e^{2f}g$, the corresponding Lee form is $\theta'=\theta-\text{d}f$. A Weyl connection is called closed (or exact) if its Lee form with respect to a metric $g\in c$ is closed (or exact). This definition is independent of the choice of $g$.

By the classification result of Merkulov-Schwachh{\"o}fer \cite{merkulovschwachhofer1999}, if $D$ is a non-closed Weyl connection on a conformal manifold $(M, c)$ of dimension different from 4, then either the holonomy of $D$ is the full conformal group $CO(n)$, or $D$ is reducible (In this case, we say that $D$ defines a conformal product structure on $(M,c)$). It is thus natural to investigate Weyl connections with reducible holonomy.

Every exact Weyl connection is the Levi-Civita connection of a metric in $c$, so this case reduces to the classical Riemannian geometry. A reducible Weyl connection which is closed but not exact is called a locally conformal product structure (LCP) \cite{flamencourt2024lcp}. The research on the LCP structure began with the counterexample given by Matveev and Nikolayevsky to the Belgun-Moroianu conjecture \cite{matveevnikolayevsky2015counterexample}. A global structure theorem has been obtained by Kourganoff \cite{kourganoff2019similarity}. A number of results on this topic have been obtained \cite{andradadelbarcomoroianu2024lcp}, \cite{delbarcomoroianu2024lie}, \cite{belgunflamencourtmoroianu2025weyl}, \cite{flamencourt2024characteristic}, \cite{moroianupilca2024adapted}. The present paper focuses on non-closed reducible Weyl connections. Before our results, conformal product structures have been studied on Riemannian manifolds with special holonomy: in \cite{belgunflamencourtmoroianu2025weyl}, \cite{MoroianuPilca2025} for the cases of K{\"a}hler manifolds, in \cite{moroianupilca2024einstein} and \cite{conformaleinstein} for the cases of Einstein manifolds, and in \cite{moroianupilca2025reducible} for the case of reducible Riemannian manifolds.

If a Weyl connection on $(M,c)$ is reducible, then the tangent bundle splits into an orthogonal direct sum of two $D$-parallel distributions $\text{T}M=T_1\oplus T_2$. The distributions $T_1$ and $T_2$ are both integrable because $D$ is torsion-free by definition. Recall that for a simply connected manifold, if the holonomy of its Levi-Civita connection is reducible, then the manifold can be decomposed into a Riemannian product of several manifolds according to the decomposition of the holonomy representation. But for conformal manifolds and Weyl connections, this type of decomposition only exists locally. A manifold with conformal product structure can be written locally as a product $M_1\times M_2$ and its Riemannian metric is of the form $e^{2f_1}g_1+e^{2f_2}g_2$, with $g_1,g_2$ Riemannian metrics on $M_1$ and $M_2$ respectively, and $f_1,f_2\in\mathcal{C}^\infty(M_1\times M_2)$, see \cite{belgunmoroianu2011weyl} or \cite{belgunflamencourtmoroianu2025weyl} for details. In the present paper, we prove that in a class of special cases, i.e., compact manifolds with constant sectional curvature and some irreducible symmetric spaces, no conformal product structure exists. 

Our main results are the following:

\begin{theorem}\label{thm1}
Let \((M,g)\) be a compact Riemannian manifold with constant sectional curvature
\(\kappa\). Then:
\begin{enumerate}
    \item If \(\kappa\neq 0\), then \((M,g)\) admits no conformal product structure.
    \item If \(\kappa=0\) and \(n\ge3\), then every conformal product structure has zero Lee form; equivalently, its Weyl connection is \(\nabla^g\).
\end{enumerate}
\end{theorem}

\begin{theorem}\label{thm2}
    Let $(M,g)$ be a compact irreducible locally symmetric space with non-positive sectional curvature. Then there is no conformal product structure on $(M,g)$. 
\end{theorem}

The rank of a conformal product structure is defined as the smallest of the ranks of the two distributions $T_1$ and $T_2$. Our argument can be divided into two cases: the rank $1$ case and the higher-rank cases, which are proved independently. In the case of rank $1$, we denote by \(\xi\) a unit vector field generating the rank-one distribution, which exists up to taking a double cover of $M$. The core of our method lies in obtaining an equation by calculating the curvature tensor $R\xi$ through two approaches, one arising from the relation between $\nabla\xi$ and $\theta$, and the other from the simple expression of Riemannian curvature on manifolds with constant sectional curvature. For the cases of rank $\ge2$, we use the involution operator $S$ introduced in \cite{MoroianuPilca2025}, which is the linear operator with $T_1$ and $T_2$ as its eigen-subspaces corresponding to the eigenvalues 1 and -1 respectively. We replace the vector $\xi$ by $S$, and the handling of this case is also divided into two parts: the case of positive (sectional) curvature and of non-positive curvature. We finally show that the method for non-positive curvature can be generalized to compact irreducible locally symmetric spaces.

\subsection*{Acknowledgments} The author would like to thank Andrei Moroianu for his guidance and support. The author is supported by a doctoral contract from Université Paris-Saclay, through the Programme blanc GS Mathématiques and the École Doctorale de Mathématiques Hadamard (EDMH).

\section{Preliminaries}
Let $(M,g)$ be a Riemannian manifold of dimension $n\ge2$. We denote by $\nabla$ the Levi-Civita connection of $g$ and by $\sharp:\text{T}^*M\rightarrow\text{T}M$ and $\flat:\text{T}M\rightarrow \text{T}^*M$ the musical isomorphisms defined by $g$, which are $\nabla$-parallel and inverse to each other. We shall freely identify 1-forms and vector fields using the metric \(g\). In order to simplify notation, we will sometimes simply write $\langle \cdot,\cdot\rangle$ instead of the metric $g$, and denote the associated norm by $\|\cdot\|$.

\begin{definition}
A Weyl connection on $(M, g)$ is a torsion-free linear connection $D$ satisfying $Dg = -2\theta \otimes g$ for some 1-form $\theta \in \Omega^1(M)$ called the Lee form of $D$ with respect to $g$.
\end{definition}

The following conformal analogue of the Koszul formula \cite{gauduchon1995weyleinstein}, shows that the Lee form uniquely determines the Weyl connection:
\begin{equation}\label{Koszul}
    D_X\xi=\nabla_X\xi+\theta(\xi)X+\theta(X)\xi-\langle X,\xi\rangle\theta^\sharp.
\end{equation}

\begin{definition}
A conformal product structure on $(M,g)$ consists of a Weyl connection $D$ together with a decomposition of the tangent bundle of $M$ as $\text{T}M=T_1\oplus T_2$, where $T_1$ and $T_2$ are orthogonal $D$-parallel non-trivial distributions. The rank of a conformal product structure is defined to be the smallest of the ranks of the two distributions $T_1$ and $T_2$. A conformal product structure $D$ on $(M,g)$ is said to be orientable if the $D$-parallel distributions $T_1$ and $T_2$ are orientable. Up to a finite cover, every conformal product structure is orientable.
\end{definition}

We now record two tensorial descriptions of conformal product structures which will be used below, according to whether the rank is \(r=1\) or \(r\ge2\).

\subsection{Conformal product structure of rank $r=1$}

\begin{lemma}\label{rank1 chara}
    If there exists a conformal product structure on $(M,g)$ of rank $1$, then after passing to a double cover if necessary, there exists a vector field $\xi$ such that $|\xi|=1$ and 
    \begin{equation}\label{rank-one-xi}
\nabla_X\xi=-\theta(\xi)X+\theta^\sharp\langle X,\xi\rangle.
    \end{equation}
\end{lemma}
\textbf{Proof.} Assume that there is a conformal product structure of rank $1$, i.e., a Weyl connection $D$ and a $D$-parallel decomposition $\text{T}M=\text{T}_1\oplus \text{T}_2$, with $1=\text{dimT}_1\le\text{dimT}_2$. Then, by replacing $M$ with a double cover if necessary, we can assume that the rank-$1$ vector bundle $T_1$ is oriented, hence it has a global section $\xi$ of unit norm, i.e., $\text{T}M=\mathbb{R}\xi\oplus\xi^\bot$.

Since $D$ preserves the decomposition $\text{T}M=\mathbb{R}\xi\oplus\xi^\bot$, there is a  1-form $\alpha$ on $M$ such that $D_X\xi=\alpha(X)\xi$ for every $X\in\text{T}M$. With the Koszul formula \eqref{Koszul}, we have 
\begin{equation}
    \alpha(X)=\langle D_X\xi,\xi\rangle=\theta(\xi)\langle X,\xi\rangle+\theta(X)\langle\xi,\xi\rangle -\langle X,\xi\rangle\theta(\xi)=\theta(X),\qquad \forall X\in\text{T}M,
\end{equation}
thus showing that $\alpha=\theta$. Hence, $\theta(X)\xi=\alpha(X)\xi=D_X\xi=\nabla_X\xi+\theta(\xi)X+\theta(X)\xi-\langle X,\xi\rangle\theta^\sharp$, thus for $\forall X\in\text{T}M$, we have: $
\nabla_X\xi=-\theta(\xi)X+\theta^\sharp\langle X,\xi\rangle$, which completes the proof of the lemma.\qed

\subsection{The involution associated with a conformal product structure.}

For any two vector fields $X,Y\in\Gamma(\text{T}M)$, we define the symmetric endomorphism:
\[
(X\odot Y)(Z):=\langle X,Z\rangle Y+\langle Y,Z\rangle X,\quad \forall Z\in\Gamma(TM),
\]
and we denote by $T:=\nabla\theta$ the bilinear form defined by the covariant derivative of the Lee form $\theta$ as well as the endomorphism determined by the metric: $g(TX,Y)=T(X,Y)=\langle \nabla_X\theta,Y\rangle$. For any 1-form \(\alpha\), we use the convention $\delta\alpha=-\sum_{i=1}^n(\nabla_{e_i}\alpha)(e_i)$, where \(\{e_i\}\) is any local
orthonormal frame. In particular,
\begin{equation}\label{trT}
\operatorname{tr}T=\sum_{i=1}^n g(Te_i,e_i)=\sum_{i=1}^n(\nabla_{e_i}\theta)(e_i)=-\delta\theta.
\end{equation}

A conformal product structure can be characterized in terms of an involution of $\text{T}M$ as follows; cf. \cite[Lemma 2.4]{MoroianuPilca2025}.

\begin{lemma}\label{tensorialcharacterization}
    On a Riemannian manifold $(M,g)$ with Levi-Civita connection $\nabla$ the following
assertions are equivalent:

(i) There exists a conformal product structure on $(M,g)$.

(ii) There exists a g-orthogonal involution $S$ of $\text{T}M$ different from $\pm\text{Id}_{\text{T}M}$ and a 1-form $\theta$ on $M$, such that
\begin{equation}\label{nabla S}
\nabla_X S=SX\odot\theta^\sharp-S\theta^\sharp\odot X,\ \ \forall X\in\Gamma(\text{T}M).
\end{equation}
\end{lemma}
For a conformal product structure, i.e., a Weyl connection $D$ and a $D-$parallel decomposition $\text{T}M=T_1\oplus T_2$, the involution $S$ is defined by $S(X_1+X_2)=X_1-X_2$, $\forall X_1\in T_1$ and $X_2\in T_2$. We assume $\text{dim}T_1=n-r$ and $\text{dim}T_2=r$ from now on, where $r$ is the rank of the conformal product structure, i.e., $\text{dim}T_1\ge\text{dim}T_2$, thus $\text{tr}S=n-2r$. We denote by \(\theta_i\) the restriction, equivalently the component, of \(\theta\) on \(T_i\). Thus \(\|\theta\|^2=\|\theta_1\|^2+\|\theta_2\|^2\) and $\langle \theta,S\theta\rangle=\|\theta_1\|^2-\|\theta_2\|^2$.

For later use, we recall the following action of the Riemannian curvature tensor on $S$ obtained in [\cite{MoroianuPilca2025}, Equation (7)], which follows directly from \eqref{nabla S} applied twice to vector fields $X,Y\in\Gamma(\text{T}M)$:
\begin{equation}\label{RS}
\begin{aligned}
    R_{X,Y}S=&SY\odot TX-SX\odot TY+STY\odot X-STX\odot Y\\ &+\langle\theta,Y\rangle(SX\odot\theta-S\theta\odot X)+\langle\theta,X\rangle(S\theta\odot Y-SY\odot\theta)\\ &-\|\theta\|^2(SX\odot Y-SY\odot X).
\end{aligned}
\end{equation}

By applying Equation \eqref{RS} to some vector field $Z\in\text{T}M$, we obtain:
\begin{equation}\label{RSapptoZ}
\begin{aligned}
(R_{X,Y}S)Z=&\langle SY,Z\rangle TX+\langle TX,Z\rangle SY-\langle SX,Z\rangle TY-\langle TY,Z\rangle SX\\
    &+\langle STY,Z\rangle X+\langle X,Z\rangle STY-\langle STX,Z\rangle Y-\langle Y,Z\rangle STX\\
    &+\langle\theta,Y\rangle \langle SX,Z\rangle\theta+\langle\theta,Y\rangle \theta(Z)SX-\theta(Y)\langle S\theta,Z\rangle X-\theta(Y)\langle X,Z\rangle S\theta\\
    &+\theta(X)\langle S\theta, Z\rangle Y+\theta(X)\langle Y,Z\rangle S\theta-\theta(X)\langle SY,Z\rangle\theta-\theta(X)\theta(Z)SY\\
    &-\|\theta\|^2\langle SX,Z\rangle Y-\|\theta\|^2\langle Y,Z\rangle SX+\|\theta\|^2\langle SY,Z\rangle X+\|\theta\|^2\langle X,Z\rangle SY.
\end{aligned}
\end{equation}

Taking $Y=Z=e_i$, summing over a local orthonormal basis $e_i$ yields:
\begin{align*}
    \sum_{i=1}^n(R_{X,e_i}S)e_i=&(trS)TX+STX-TSX-(trT)SX+tr(ST)X\\
    &+STX-STX-nSTX+\theta(SX)\theta+\|\theta\|^2SX\\
    &-(\|\theta_1\|^2-\|\theta_2\|^2)X-\theta(X)S\theta+\theta(X)S\theta+n\theta(X)S\theta-(trS)\theta(X)\theta\\
    &-\theta(X)S\theta-\|\theta\|^2SX-n\|\theta\|^2SX+\|\theta\|^2tr(S)X+\|\theta\|^2SX.
\end{align*}
After simplification and using \eqref{trT}, we obtain the following equation for $\forall X\in\text{T}M$:
\begin{equation}\label{centre}
\begin{aligned}
    \sum_{i=1}^n(R_{X,e_i}S)e_i=&(n-2r)TX-TSX+(\delta\theta)SX+tr(ST)X\\
    &-(n-1)STX+\theta(SX)\theta-(\|\theta_1\|^2-\|\theta_2\|^2)X\\
    &+(n-1)\theta(X)S\theta-(n-2r)\theta(X)\theta\\
    &-(n-1)\|\theta\|^2SX+\|\theta\|^2(n-2r)X.
\end{aligned}
\end{equation}

We will need the following simple lemma to compute $R_{X,Y}S$ directly:

\begin{lemma}\label{trivial}
For $\forall X,Y,Z\in\Gamma(\text{T}M)$, 
\[
(R_{X,Y}S)Z=R_{X,Y}(SZ)-S(R_{X,Y}Z).
\]
\end{lemma}
\textbf{Proof:}
\[
    (R_{X,Y}S)Z=((\nabla_X\nabla_Y-\nabla_Y\nabla_X-\nabla_{[X,Y]})S)(Z).
\]
We compute the three terms on the right hand side. The first term is:
\begin{align*}
    (\nabla_X\nabla_YS)Z&=\nabla_X((\nabla_Y S)Z)-(\nabla_Y S)(\nabla_XZ)\\
    &=\nabla_X\nabla_Y(SZ)-\nabla_X(S(\nabla_YZ))-(\nabla_YS)(\nabla_XZ)\\
    &=\nabla_X\nabla_Y(SZ)-(\nabla_XS)(\nabla_YZ)-S(\nabla_X\nabla_YZ)-(\nabla_YS)(\nabla_XZ).
\end{align*}
By exchanging $X$ and $Y$, we obtain the second term:
\[
    (\nabla_Y\nabla_XS)Z=\nabla_Y\nabla_X(SZ)-(\nabla_YS)(\nabla_XZ)-S(\nabla_Y\nabla_XZ)-(\nabla_XS)(\nabla_YZ).
\]
The third term is:
\[
    (\nabla_{[X,Y]}S)Z=\nabla_{[X,Y]}(SZ)-S(\nabla_{[X,Y]}Z).
\]
Adding the above three equations together yields:
\begin{equation*}
\begin{aligned}
   (R_{X,Y}S)Z&= (\nabla_X\nabla_YS-\nabla_Y\nabla_XS-\nabla_{[X,Y]}S)Z\\
   &=R_{X,Y}(SZ)-S(R_{X,Y}Z).\\
   \end{aligned}
\end{equation*} 
This concludes the proof.\qed

\section{Main Results}

In this section, we present the proofs of our main results, Theorem \ref{thm1} and \ref{thm2}, which were stated in the Introduction.

\subsection{The case of constant sectional curvature surfaces}

\begin{proposition}\label{prop:surface-case}
Let \((M,g)\) be a compact connected surface of non-zero constant sectional curvature. Then \((M,g)\) admits no conformal product structure.
\end{proposition}

\begin{proof}



Assume that \((M,g)\) admits a conformal product structure. Since \(\dim M=2\),
both distributions have rank one. After passing to a finite cover which is
orientable and on which one of the line distributions is orientable, we obtain
a compact orientable surface carrying a nowhere vanishing vector field. Hence its
Euler characteristic is zero by the Poincaré--Hopf theorem.

On the other hand, the lifted metric has constant Gaussian curvature
\(\kappa\neq0\). By Gauss--Bonnet,
\[
2\pi\chi(\widetilde M)
=
\int_{\widetilde M}\kappa\,d\mathrm{vol}_{\tilde g}
=
\kappa\,\mathrm{Vol}(\widetilde M)
\neq0,
\]
a contradiction.
\end{proof}

\begin{remark}
Notice that Theorem \ref{thm1} makes no assertion about flat surfaces. Now we give a converse example in this case: let $T^2=\mathbb R^2/\mathbb Z^2$ be endowed with its standard flat metric, and let \(E_1,E_2\) be the global parallel orthonormal frame induced by the coordinate vector fields. We define the parallel complex structure \(J\) by $J E_1 = E_2$ and $J E_2 = -E_1$. Let \(f\in C^\infty(T^2)\), and set
\[
        \xi := \cos f\, E_1 + \sin f\, E_2, \qquad
        \eta := J\xi .
\]
Then \(\xi,\eta\) form a global orthonormal frame of \(T T^2\), and hence define an orthogonal splitting $\text{T}T^2 = \mathbb R \xi \oplus \mathbb R \eta$. Since \(E_1,E_2\) are parallel and \(J\) is parallel, one has $\nabla_X \xi = X(f)\eta$ for every vector field \(X\):
\[
        \begin{aligned}
        \nabla_X\xi&=\nabla_X(\cos fE_1+\sin fE_2)\\
        &=X(\cos f)E_1+X(\sin f)E_2\\
        &=-\sin fX(f)E_1+\cos fX(f)E_2\\
        &=X(f)(-\sin fE_1+\cos fE_2)=X(f)\eta.\\
        \end{aligned}
\]
Applying $J$ to both sides of the equation, we obtain $\nabla_X\eta=-X(f)\xi$, because \(J\) is parallel and \(\eta=J\xi\). Now define the Lee form \(\theta\) by $\theta^\sharp := J\nabla f$. Writing \(X=u\xi+v\eta\), we have $\theta^\sharp=\xi(f)\eta-\eta(f)\xi$, therefore $\theta(\xi)=-\eta(f)$ and $\theta(\eta)=\xi(f)$.  Moreover, using the Koszul formula \eqref{Koszul}, we first compute:
\[
        D_X\xi=X(f)\eta-\eta(f)(u\xi+v\eta)+\theta(X)\xi-u\bigl(\xi(f)\eta-\eta(f)\xi\bigr)=\theta(X)\xi .
\]
Indeed, the \(\eta\)-component cancels because
\[
        X(f)=u\xi(f)+v\eta(f).
\]
Similarly,
\[
        D_X\eta=-X(f)\xi+\xi(f)(u\xi+v\eta)+\theta(X)\eta-v\bigl(\xi(f)\eta-\eta(f)\xi\bigr)=\theta(X)\eta .
\]
Hence $D_X\xi\in \mathbb R\xi$ and $D_X\eta\in \mathbb R\eta$ for every vector field \(X\). Thus both line distributions $\mathbb R\xi$ and $\mathbb R\eta$ are \(D\)-parallel, hence the splitting $\text{T}T^2=\mathbb R\xi\oplus\mathbb R\eta$ defines a conformal product structure.

Moreover, in the coordinates \(x,y\) corresponding to \(E_1,E_2\), we have $\theta = -f_y\,dx+f_x\,dy$, and consequently
\[
        d\theta = (f_{xx}+f_{yy})\,dx\wedge dy .
\]
Thus \(\theta\) is not closed in general. 
\end{remark}

\subsection{The case of constant sectional curvature in dimension $\ge 3$}
\begin{theorem}\label{3.1}
    Let $(M,g)$ be a compact manifold with constant sectional curvature of dimension $n\ge 3$. Then $(M,g)$ admits no conformal product structure $D$, unless $g$ is flat and $D=\nabla^g$.
\end{theorem}

\textbf{Proof.} We denote the sectional curvature of $(M,g)$ by $\kappa$. The proof splits into two cases: the case of $\kappa> 0$ and of $\kappa\le 0$.

\textbf{Case i)} $\kappa>0$.
We first deal with the case where the rank is $1$.

\begin{proposition}\label{prop2} Let $M$ be a compact Riemannian manifold of constant sectional curvature $\kappa>0$ with $\text{dim}M\ge 3$. Then $M$ admits no conformal product structures of rank $1$.
\end{proposition}

\textbf{Proof.}
Suppose, by contradiction, that \(M\) admits a rank-one conformal product structure. Passing to a double cover if necessary, we may assume that the rank-one distribution is orientable. This does not affect the argument, since a finite cover of a compact manifold of constant sectional curvature \(\kappa>0\) is again compact and has the same constant sectional curvature. By Lemma \ref{rank1 chara}, there exists a global unit vector field \(\xi\) spanning the rank-one
distribution such that $\nabla_X\xi=-\theta(\xi)X+\theta^\sharp\langle X,\xi\rangle$ for every \(X\in TM\).

We first consider the case of $\text{dim}M\ge 3$. A Riemannian manifold $(M,g)$ has constant sectional curvature $\kappa$ if and only if
\begin{equation}\label{curvature}
    R(X,Y)Z=\kappa(g(Z,Y)X-g(Z,X)Y)
\end{equation}
Using \eqref{rank-one-xi} twice, we compute: 
\begin{align*}
    R_{X,Y}\xi&=\nabla_X\nabla_Y\xi-\nabla_Y\nabla_X\xi-\nabla_{[X,Y]}\xi\\
    &=-X(\theta(\xi))Y+Y(\theta(\xi))X+TX\langle Y,\xi\rangle-TY\langle X,\xi\rangle+\theta^\sharp\langle Y,\nabla_X\xi\rangle-\theta^\sharp\langle X,\nabla_Y\xi\rangle
\end{align*}
Applying \eqref{curvature} to $Z=\xi$, we have:
\begin{equation}\label{11}
\begin{aligned}
    0=&\kappa\langle Y,\xi\rangle X-Y(\theta(\xi))X-\kappa\langle X,\xi\rangle Y+X(\theta(\xi))Y-TX\langle Y,\xi\rangle \\
    &+TY\langle X,\xi\rangle -\theta^\sharp\langle Y,\nabla_X\xi\rangle +\theta^\sharp\langle X,\nabla_Y\xi\rangle.
\end{aligned}
\end{equation}
For $Y=\xi$, we get:
\[
    0=\kappa X-\xi(\theta(\xi))X-\kappa\langle X,\xi\rangle\xi+X(\theta(\xi))\xi-TX+T\xi\langle X,\xi\rangle +\theta^\sharp\langle X,\nabla_\xi\xi\rangle.
\]
We denote $X'=X-\langle X,\xi\rangle\xi$, so that $\langle X',\xi\rangle=0$. The previous equation then reads:
\begin{equation}\label{12}
    TX'=\kappa X'-\xi(\theta(\xi))X+X(\theta(\xi))\xi+\theta^\sharp\langle X,\nabla_\xi\xi\rangle.
\end{equation}
Taking the scalar product with $\theta$ in this equation yields:
\begin{equation}\label{9}
    \frac12X'\|\theta\|^2=\kappa\theta(X')-\xi(\theta(\xi))\theta(X)+X(\theta(\xi))\theta(\xi)+\langle X,\nabla_\xi\xi\rangle\|\theta\|^2.
\end{equation}
Consider the decomposition of $\theta=\theta_0+a\xi^\flat$, where $\theta_0$ is the  restriction of $\theta$ on $\xi^\bot$, and $a=\theta(\xi)$. We have $\|\theta\|^2=\|\theta_0\|^2+a^2$ and  $\theta(X)=\theta_0(X')+a\langle X,\xi\rangle$. Equation \eqref{9} can be written as:
\begin{equation}
\frac12X'(\|\theta_0\|^2+a^2)=\kappa\theta_0(X')-\xi(a)\theta(X)+X(a)a+\langle X',\nabla_\xi\xi\rangle\|\theta\|^2
\end{equation}
or, equivalently: 
\begin{align*}
    \frac12X'\|\theta_0\|^2&=\kappa\theta_0(X')-\xi(a)\theta(X)+\langle X,\xi\rangle\xi(a)a+\langle X',\nabla_\xi\xi\rangle\|\theta\|^2\\
    &=\kappa\theta_0(X')-\xi(a)\theta(X')+\langle X',\nabla_\xi\xi\rangle\|\theta\|^2
\end{align*}
because $-\theta(X)+\langle X,\xi\rangle a=-\theta(X')\ \  \forall X'\in\xi^\bot$. We thus obtained:
\begin{equation}\label{cle}
    \frac12 \pi(\text{d}\|\theta_0\|^2)=\kappa\theta_0-\xi(a)\theta_0+\nabla_\xi\xi\|\theta\|^2
\end{equation}
where $\pi$ is the natural projection $TM\rightarrow\xi^\bot$.

A direct computation yields:
\[
\nabla_\xi\xi=-\theta(\xi)\xi+\theta^\sharp\langle\xi,\xi\rangle=-\theta(\xi)\xi+\theta^\sharp=\theta_0,\quad \nabla_{\theta^\sharp}\xi=-\theta(\xi)\theta^\sharp+\theta^\sharp\theta(\xi)=0.
\]
If $X,Y\in\xi^\bot$, then Equation \eqref{11} becomes:
\begin{equation}\label{XY bot xi}
    0=-Y(a)X+X(a)Y-\theta^\sharp\langle Y,\nabla_X\xi\rangle+\theta^\sharp\langle X,\nabla_Y\xi\rangle,
\end{equation}
and $\nabla_X\xi=-\theta(\xi)X+\theta^\sharp\langle X,\xi\rangle=-\theta(\xi)X$. So by \eqref{XY bot xi}, we have: 
\begin{equation}
    0=-Y(a)X+X(a)Y+\theta^\sharp\langle Y,\theta(\xi)X\rangle-\theta^\sharp\langle X,\theta(\xi)Y\rangle=-Y(a)X+X(a)Y.
\end{equation}
Thus $Y(a)X=X(a)Y$, for every $X,Y\in\xi^\bot$. Since we have assumed that $M$ is of dimension $n\ge 3$, we can take $X$ and $Y$ to be linearly independent, we obtain $Y(a)=X(a)=0$, which means
\begin{equation}\label{a-basic}
X(a)=0,\qquad \forall X\in\xi^\bot.
\end{equation}
Now we assume $X\bot \xi$ so that $X'=X$. Then Equation \eqref{12} can be written as: 
\begin{equation}\label{TX}
    TX=\kappa X-\xi(a)X+X(a)\xi+\theta^\sharp\langle X,\nabla_\xi\xi\rangle=\kappa X-\xi(a)X+\theta^\sharp\theta_0(X).
\end{equation}
The scalar product of \eqref{cle} with $\theta_0$ is:
\begin{equation}\label{theta0}
    \frac12\theta_0(\|\theta_0\|^2)=\|\theta_0\|^2(\kappa-\xi(a)+\|\theta\|^2).
\end{equation}
Since $X(a)=0$ for $\forall X\bot\xi$, $a$ is a basic function of the foliation $F$ associated with the distribution $\xi^\bot$, i.e., the restriction of $a$ on an arbitrary leaf of $F$ is constant.  Let \(x\in M\) be a maximum or minimum point of \(a\). Since \(a\) is constant along the leaf \(L_x\), every point of \(L_x\) is again a maximum or minimum point of \(a\). Hence \(da=0\) on \(L_x\), and in particular
\(\xi(a)=0\) on \(L_x\). By continuity, \(\xi(a)=0\) on \(\overline{L_x}\). Since \(\theta_0^\sharp\) is tangent to the leaves of \(\xi^\perp\), the closure \(\overline{L_x}\) is invariant under the local flow of \(\theta_0^\sharp\). Let \(y\in\overline{L_x}\) be a maximum point of \(\|\theta_0\|^2\) on \(\overline{L_x}\) (such a point exists because $\overline{L_x}$ is compact). Then $\theta_0(\|\theta_0\|^2)(y)=0$. Hence \eqref{theta0} gives $0=\|\theta_0\|^2(y)(\kappa+\|\theta\|^2(y))$. Since \(\kappa>0\), this implies \(\|\theta_0\|^2(y)=0\), therefore \(\theta_0\equiv0\) on \(\overline{L_x}\), which means that $\theta=a\xi^\flat$ on $L_x$. 


Moreover, since we have assumed $X\bot\xi$, i.e., $X$ is tangent to $L_x$, we know that $\nabla_X \theta=\nabla_X(a\xi)$ on $L_x$.

With \eqref{rank-one-xi} and \eqref{a-basic} we know at each point of $L_x$, that:
\begin{equation}
TX=\nabla_X\theta=\nabla_X(a\xi)=X(a)\xi+a\nabla_X\xi=-a^2X.
\end{equation}
Using \eqref{TX}, and the fact that $\theta_0\equiv 0$ on $L_x$ we get: 
\begin{equation}
    -a^2X=TX=\kappa X-\xi(a)X+\theta^\sharp\theta_0(X)=\kappa X-\xi(a)X.
\end{equation}
for every $X\in\xi^{\perp}$.

Then $\kappa+a^2-\xi(a)=0$, at points maximum or minimum of $a$ , $\xi(a)=0$, so $\kappa+a^2=0$, which contradicts the assumption that $\kappa>0$. This completes the proof of Proposition \ref{prop2}
\qed

We now consider the case $\kappa>0$ and the rank of conformal product structure $\ge2$.

As Equation \eqref{centre} holds for any $X\in\text{T}M$, we can rewrite it as an identity in $\text{End}(\text{T}M)$ as follows:
\begin{equation}\label{maineq}
\begin{aligned}
    \kappa[(n-2r)\operatorname{Id}-nS]=&(n-2r)T-TS+(\delta\theta)S+\operatorname{tr}(ST)\operatorname{Id}\\
    &-(n-1)ST+\theta\otimes S\theta-(\|\theta_1\|^2-\|\theta_2\|^2)\operatorname{Id}\\
    &+(n-1)S\theta\otimes \theta-(n-2r)\theta\otimes\theta\\
    &-(n-1)\|\theta\|^2S+(n-2r)\|\theta\|^2\operatorname{Id}.
\end{aligned}
\end{equation}
 The endomorphisms of $\text{T}M$ can be decomposed into a symmetric part and a skew-symmetric part with respect to the metric $g$ as follows. For any $U\in\text{End}(\text{T}M)$, we denote by $U^t$ its adjoint with respect to $g$. Then:  
\[
U=\frac12(U+U^t)+\frac12(U-U^t),
\]
where the first term is the symmetric part and the second term is the skew-symmetric part. This decomposition is unique, since an endomorphism which is both symmetric and skew-symmetric must be zero, i.e. the vector space of symmetric endomorphisms and that of skew-symmetric endomorphisms are orthogonal.

We denote the skew-symmetric part of $T$ as $T^a:=\frac12(T-T^t)$. We consider the skew-symmetric part of \eqref{maineq}:
\begin{align}\label{anti-sym eqmain}
    0=2(n-2r)T^a-(TS-ST^t)-(n-1)(ST-T^tS)-(n-2)S\theta\wedge\theta.
\end{align}

\begin{lemma}
Under the standard identification between 2-forms and skew-symmetric endomorphisms, we have:  $T^a=\frac12\text{d}\theta$
\end{lemma}
\textbf{Proof.}
\begin{align*}
\langle T^aX,Y\rangle&=\frac12(\langle TX,Y\rangle-\langle T^tX,Y\rangle)=\frac12(\langle\nabla_X\theta,Y\rangle-\langle X,\nabla_Y\theta\rangle)\\
&=\frac12[(X(\theta(Y))-\theta(\nabla_XY))-(Y(\theta(X))-\theta(\nabla_YX))]\\
&=\frac12(X(\theta(Y))-Y(\theta(X))-\theta([X,Y]))\\
&=\frac12\text{d}\theta(X,Y).
\end{align*}\qed

We introduce another decomposition of endomorphism of $\text{T}M$, into $S$-commuting part and $S$-anti commuting part as follows. For every $U\in\text{End}(\text{T}M)$, we write:
\[
U=\frac12(U+SUS)+\frac12(U-SUS).
\]
Since $S^2=\text{Id}$, the first term on the right hand side commutes with $S$, and the second term anti-commutes with $S$. If an endomorphism $U$ both commutes and anti-commutes with $S$, then $US=SU=-US$, whence $US=0$, so $U=0$ as $S$ is invertible. So the decomposition is unique.

We consider the $S$-(anti)commuting decomposition of the \eqref{anti-sym eqmain}. We separately calculate the decomposition of the first three terms in \eqref{anti-sym eqmain}, where the first terms commutes with $S$ and second terms anti-commutes with $S$:
\[
2T^a=(T^a+ST^aS)+(T^a-ST^aS)
\]
\[
TS-ST^t=(T^a+T^s)S-S(T^s-T^a)=(T^aS+ST^a)+(T^sS-ST^s)
\]
\[
ST-T^tS=S(T^a+T^s)-(T^s-T^a)S=(ST^a+T^aS)-(T^sS-ST^s)
\]
The fourth term of \eqref{anti-sym eqmain} is $S$-anti-commuting. Indeed:
\begin{align*}
    (S\theta\wedge\theta)(SX)&=\langle S\theta,SX\rangle\theta-\langle \theta,SX\rangle S\theta=\langle\theta,X\rangle \theta-\langle \theta, SX\rangle S\theta\\
    S((S\theta\wedge\theta)X)&=S(\langle S\theta,X\rangle \theta-\langle \theta,X\rangle S\theta)=\langle S\theta,X\rangle S\theta-\langle \theta,X\rangle\theta,
\end{align*}
so its $S$-commuting part is zero. 

 Consequently the $S$-anti-commuting part of \eqref{maineq} is 
\begin{equation}\label{acpart1}
0=(n-2r)(T^a-ST^aS)+(n-2)(T^sS-ST^s)-(n-2)S\theta\wedge \theta.
\end{equation}
On the other hand, the composition of \eqref{maineq} with $S$ is: 
\begin{equation}\label{Smaineq}
\begin{aligned}
\kappa[(n-2r)S-n\text{Id}]=&(n-2r)ST-STS+\delta(\theta)\text{Id}+\text{tr}(ST)S\\
&-(n-1)T+S\theta\otimes S\theta-(\|\theta_1\|^2-\|\theta_2\|^2)S\\
&+(n-1)\theta\otimes \theta-(n-2r)S\theta\otimes\theta\\
&-(n-1)\|\theta\|^2\text{Id}+(n-2r)\|\theta\|^2S.
\end{aligned}
\end{equation}
We will use the same decomposition as above for this equation.
Its skew-symmetric part is: 
\[
0=(n-2r)(ST-T^tS)-S(T-T^t)S-(n-1)(T-T^t)-(n-2r)\theta\wedge S\theta.
\]
And the $S$-anti-commuting part is: 
\begin{equation}\label{acpart2}
0=(n-2r)(ST^s-T^sS)-(n-2)(T^a-ST^aS)-(n-2r)\theta\wedge S\theta.
\end{equation}
By combining Equations \eqref{acpart1} and \eqref{acpart2}, 
we define $P=ST^s-T^sS-\theta\wedge S\theta$ and $Q=ST^aS-T^a$, and obtain:
\[
\begin{cases}
    (n-2)P+(n-2r)Q=0\\
    (n-2r)P+(n-2)Q=0.
\end{cases}
\]

Since we have already assumed $r\ge 2$, the above system admits a unique solution, namely $P=Q=0$.

On the other hand, the $S$-commuting part of \eqref{anti-sym eqmain} is 
\begin{equation}\label{1}
0=(n-2r)(T^a+ST^aS)-n(ST^a+T^aS).
\end{equation}
We compose both sides of this identity on the left with $S$, and obtain
\begin{equation}\label{2}
0=(n-2r)(ST^a+T^aS)-n(T^a+ST^aS).
\end{equation}
By combining \eqref{1} and \eqref{2}, we deduce $ST^a+T^aS=0$.

Since $Q=ST^aS-T^a=0$, we also have $ST^a-T^aS=0$, thus $\text{d}\theta=2T^a=0$.

If the manifold $M$ is simply connected, every closed 1-form $\theta$ is exact, so we can choose a metric $g'$ conformal to $g$ such that the Weyl connection $D$ is Levi-Civita connection of $g'$. Since $M$ is complete, it can be decomposed into a Riemannian product by the decomposition theorem of de Rham. On the other hand, any simply connected manifold with constant positive sectional curvature is diffeomorphic to the sphere $S^n$ (see Section 12.3 of \cite{petersenriemanniangeometry}).
\begin{lemma}
     The topological manifold $S^n$ cannot be decomposed into a product of two topological manifolds of positive dimension.
\end{lemma}
\textbf{Proof.}
Assume that there exist two manifolds $M$ and $N$ such that $S^n=M\times N$, with $p:=\text{dim}M\ge 1$ and $q:=\text{dim}N\ge 1$. By the K\"{u}nneth formula, the $k-$th homology of $S^n$ can be decomposed as:
\[
H_k(S^n,\mathbb{Z}_2)=H_k(M\times N,\mathbb{Z}_2)\simeq \bigoplus_{i+j=k}H_i(M,\mathbb{Z}_2)\otimes H_j(N,\mathbb{Z}_2).
\]
Since \(S^n=M\times N\) is compact and connected, both \(M\) and \(N\) are compact and connected. Thus $H_p(M,\mathbb{Z}_2)=\mathbb{Z}_2$, and $H_0(N,\mathbb{Z}_2)=\mathbb{Z}_2$. Thus $H_p(M;\mathbb Z_2)\otimes H_0(N;\mathbb Z_2)\cong \mathbb Z_2$ is a non-zero direct summand of \(H_p(M\times N;\mathbb Z_2)\). Hence $H_p(M\times N;\mathbb Z_2)\neq0$. But we know the singular homology of $S^n$ is :
\begin{equation}
H_k(S^n,\mathbb{Z}_2)=
\begin{cases}
    \mathbb{Z}_2,\ \ k=0,n\\
    0,\ \ \text{otherwise}.
\end{cases}
\end{equation}
So $p=0$ or $p=n$, which contradicts the assumption $p\ge 1$ and $q=n-p\ge 1$. This proves the lemma. \qed

If $(M,g)$ is not simply connected, we set $(\tilde{M},\tilde{g})$ its universal cover manifold, with Riemannian metric $\tilde{g}=\pi^*g$. Therefore its sectional curvature is the constant $\kappa>0$. Since $M$ is compact, $(M,g)$ is complete, so $(\tilde{M},\tilde{g})$ is complete as well (cf. \cite{KN1} Ch IV, Thm 4.6(2)).  Then by Myers' theorem, we know that $\tilde{M}$ is compact. If there is a conformal product structure on $M$, it can be pulled back to $\tilde{M}$ and thereby reducing to the simply connected case. This concludes the proof of the case $\kappa>0$.

\textbf{Case ii)} $\kappa\le 0$.
Assume that $M$ admits a conformal product structure: $\text{T}M=T_1\oplus T_2$, with $\text{dim}T_1=n-r$, $\text{dim}T_2=r$. Let $\theta_1=\theta|_{T_1}$ and $\theta_2=\theta|_{T_2}$.

Using \eqref{curvature}, we obtain
\begin{equation}
\begin{aligned}
(R_{X,Y}S)Z&=R_{X,Y}(SZ)-S(R_{X,Y}Z)\\
   &=-\kappa[\iota_{S(Z)}X\wedge Y+S(\iota_ZX\wedge Y)]\\
   &=\kappa[-\langle SZ,X\rangle Y+X\langle SZ,Y\rangle+\langle Z,X\rangle SY-\langle Z,Y\rangle SX].
\end{aligned}
\end{equation}
Taking $Y=Z=e_i$, summing over a local orthonormal basis $\{e_i\}$.
\[
\sum_{i=1}^n(R_{X,e_i}S)e_i=\kappa[(n-2r)X-nSX]
\]
Combining this with \eqref{centre}, we obtain:
\begin{equation}
    \begin{aligned}
        \kappa[(n-2r)X-nSX]=&(n-2r)TX-TSX+(\delta\theta)SX+tr(ST)X\\
    &-(n-1)STX+\theta(SX)\theta-(\|\theta_1\|^2-\|\theta_2\|^2)X\\
    &+(n-1)\theta(X)S\theta-(n-2r)\theta(X)\theta\\
    &-(n-1)\|\theta\|^2SX+\|\theta\|^2(n-2r)X.
    \end{aligned}
\end{equation}
Taking the scalar product with $SX$ in this formula, and summing over a local orthonormal basis $X=e_i$ yields:
\begin{equation}\label{sumRS}
\begin{aligned}
    \kappa[(n-2r)^2-n^2]=&(n-2r)tr(ST)-trT+n\delta\theta+(n-2r)tr(ST)-(n-1)trT\\
    &+\|\theta
\|^2-(n-2r)(\|\theta_1\|^2-\|\theta_2\|^2)+(n-1)\|\theta\|^2\\
&-(n-2r)(\|\theta_1\|^2-\|\theta_2\|^2)-n(n-1)\|\theta\|^2+(n-2r)^2\|\theta\|^2.
\end{aligned}
\end{equation}
With Lemma 4.2 of \cite{MoroianuPilca2025}, we have the expression of $tr(ST)$:
\begin{equation}
tr(ST)=-\delta(S\theta)-trS\|\theta\|^2+n\langle\theta,S\theta\rangle.
\end{equation}
We substitute this expression of $\text{tr}(ST)$ into the above equation, and obtain:
\begin{align*}
    \kappa[-4nr+4r^2]&=2(n-2r)[-\delta(S\theta)-(n-2r)\|\theta\|^2+n(\|\theta_1\|^2-\|\theta_2\|^2)]\\&+(n+1)\delta\theta
    +(n-1)\delta\theta+\|\theta\|^2-(2n-4r)(\|\theta_1\|^2-\|\theta_2\|^2)\\
    &-(n-1)^2\|\theta\|^2+(n-2r)^2\|\theta\|^2.
\end{align*}
 By substituting $\|\theta\|^2 = \|\theta_1\|^2 + \|\theta_2\|^2$, the non-divergence terms simplify significantly: 
 \begin{equation}\label{H}
 -4\kappa r(n-r) = 2n\delta\theta - 2(n-2r)\delta(S\theta) - 4r(r-1)\|\theta_1\|^2 - 4(n-r)(n-r-1)\|\theta_2\|^2 
 \end{equation}
 Dividing both sides by $-4$, we obtain the pointwise algebraic identity on $M$:
 \[
 \kappa r(n-r)=-\frac{n}{2}\delta\theta + \frac{n-2r}{2}\delta(S\theta)+r(r-1)\|\theta_1\|^2 + (n-r)(n-r-1)\|\theta_2\|^2
 \]

By integrating this equation over the compact manifold $M$ against the Riemannian volume form $\text{dvol}_g$, the integrals of the divergence terms $\delta(S\theta)$ and $\delta\theta$ vanish by Stokes' Theorem, we obtain:
\begin{equation}\label{int result}
    (r-1)r\int_M\|\theta_1\|^2\text{dvol}_g+(n-r)(n-r-1)\int_M\|\theta_2\|^2\text{dvol}_g=\kappa r(n-r)\text{Vol}(M).
\end{equation}

We have already assumed that the curvature $\kappa$ of $M$ is nonpositive, so if $\kappa<0$, this yields a contradiction, and if $\kappa=0$, there are two cases:

\textbf{i)} If $r\ge 2$, \eqref{int result} leads to $\theta=0$. 

\textbf{ii)} 
If \(r=1\), let \(\xi\) be a unit section of the rank-one distribution \(T_2\)
after passing to a double cover if necessary, and write $\theta=\theta_0+a\xi^\flat$, where \(\theta_0\) is the component of \(\theta\) along \(\xi^\perp=T_1\) and \(a=\theta(\xi)\). From \eqref{int result} we get \(\theta_2=0\), hence \(a=0\). Therefore \(\xi(a)=0\). Applying \eqref{theta0} with \(\kappa=0\), we obtain
\[
\frac12\theta_0(\|\theta_0\|^2)=\|\theta_0\|^4.
\]
At a maximum point of \(\|\theta_0\|^2\), the left-hand side vanishes, and hence
\(\|\theta_0\|=0\) there. Thus \(\theta_0\equiv0\), and consequently
\(\theta=0\).
\qed

Combining Proposition \ref{prop:surface-case} with Theorem \ref{3.1}, we obtain Theorem \ref{thm1}.  

\subsection{Irreducible locally symmetric spaces}
We now address the second main result of this paper, Theorem \ref{thm2}. In what follows, we demonstrate that the methods used above for constant non-positive curvature can be adapted to irreducible compact locally symmetric spaces.

To prove Theorem \ref{thm2}, we need some classical results about the symmetric spaces (cf. Ch XI of \cite{KN2}):
\begin{proposition}\label{sym} Let $M$ be a Riemannian symmetric space, $G=\text{Isom}^\circ(M)$ the identity component of its isometry group, $\mathfrak{g}$ its Lie algebra and $B$ its Killing form, $o\in M$ and $\mathfrak{g}=\mathfrak{p}\oplus\mathfrak{k}$ the associated Cartan decomposition. Let $\phi_o : G \to M$ be the orbit map defined by $\phi_o(g) = g.o$, whose differential at the identity $e\in G$, $d_e\phi_o : \mathfrak{g} \to T_oM$, restricts to a linear isomorphism from $\mathfrak{p}$ onto $T_oM$. For $X,Y,Z\in\mathfrak{p}$, we set $u=d_e\phi_o(X),v=d_e\phi_o(Y)$ and $w=d_e\phi_o(Z)\in T_oM$. We then have:
\begin{equation}
R_o(u,v)w=-d_e\phi_o([[X,Y],Z]).
\end{equation}
\end{proposition} 

\textbf{Remark.} If $M$ is irreducible, then there is $\alpha\in\mathbb{R}$ such that $B(X,Y)=\alpha\langle X,Y\rangle$ for all $X,Y\in\mathfrak{p}$. According to the classification of irreducible Riemannian symmetric spaces, $(M, g)$ falls into one of two categories: compact type ($\alpha < 0$), non-compact type ($\alpha > 0$). If $(u,v)$ is an orthonormal basis of a plane $P\subset T_oM$, $K(P)$ is its sectional curvature, and $X,Y\in \mathfrak{p}$ are such that $d_e\phi_o(X)=u$ and $d_e\phi_o(Y)=v$, then:
\begin{equation}
\begin{aligned}
 K(P)&=R_o(u,v,v,u)=\langle R_o(u,v)v|u\rangle\\
 &=\langle -[[X,Y],Y]|X\rangle=\frac{1}{\alpha}B([Y,[X,Y]],X)\\
 &=\frac{1}{\alpha}B([X,Y],[X,Y]).
\end{aligned}
\end{equation}

Under the assumptions of Theorem \ref{thm2}, let $(M, g)$ be a compact irreducible locally symmetric space with non-positive sectional curvature. Since the Euclidean type is excluded, its universal cover $\tilde{M}$ must be an irreducible global Riemannian symmetric space of non-compact type (i.e., $\alpha >0$). The local calculations of Riemannian curvature tensor can be lifted to $\tilde{M}$, and in this case, the Killing form $B$ is positive definite on $\mathfrak{p}$ and negative definite on $\mathfrak{k}$.

\textbf{Proof of Theorem \ref{thm2}.} 
For an arbitrary point $o\in M$, choose a lift \(\tilde o\in\widetilde M\). For the following pointwise curvature computation, we lift all objects to the universal cover \(\widetilde M\), which is a globally symmetric space of non-compact type, and omit the tildes from the notation. Since the quantities involved are pullbacks of tensors on \(M\), the resulting value of \(H(o)\) is independent of the chosen lift \(\tilde o\). Let $\mathfrak{g}=\mathfrak{p}\oplus\mathfrak{k}$ be the Cartan decomposition at \(\tilde o\), and identify \(\mathfrak p\) with \(T_{\tilde o}\widetilde M\) via the differential of the orbit map. Choose an orthonormal basis \(e_1,\ldots,e_n\) of \(T_{\tilde o}\widetilde M\), adapted to the lifted splitting. Via the
identification \(\mathfrak p\simeq T_{\tilde o}\widetilde M\), we use the same symbols \(e_i\) for the corresponding elements of \(\mathfrak p\).


As in the proof of the non-positive constant sectional curvature case, we denote $H(o):=\sum_{1\le i,j\le n}\langle R_{e_j,e_i}(S)e_i,Se_j\rangle$. We will compute $H(o)$ by two methods: one using Proposition \ref{sym} and on the other hand, using Equation \eqref{RS}. Note that the expression $H(o)=\sum_{1\le i,j\le n}\langle R_{e_j,e_i}(S)e_i,Se_j\rangle$ does not depend on the choice of the orthonormal basis $\{e_i\}$, so $H$ is a well-defined function on $M$.

In order to simplify the computation, we assume that $T_1$ is generated by $e_i$ for $i=1,\dots, n-r$ and $T_2$ is generated by $e_i$ for $i=n-r+1,\dots,n$. So $Se_i=e_i$ for $1\le i\le n-r$ and $Se_i=-e_i$ for $n-r+1\le i\le n$. For \(A\in\mathfrak k\), we set $
\|A\|_{\mathfrak k}^2:=-B(A,A)$. We omit the point $(o)$ where the computation is done. In the following computation, we utilize the Jacobi identity for the term $[[e_j, e_i], Se_i]$. Since $e_i$ is an eigenvector of $S$, $e_i$ and $Se_i$ are collinear, yielding $[e_i, Se_i] = 0$. Combined with the ad-invariance of the Killing form $B([X, Y], Z) = -B(Y, [X, Z])$, we obtain:
\begin{align*}
    H&=\sum_{1\le i,j\le n}\langle R_{e_j,e_i}(S)e_i,Se_j\rangle\\
    &=\sum_{1\le i,j\le n}\langle R_{e_j,e_i}(Se_i),Se_j\rangle-\langle SR_{e_j,e_i}e_i,Se_j\rangle\\
    &=\sum_{1\le i,j\le n}\langle-[[e_j,e_i],Se_i],Se_j\rangle-\langle-S([[e_j,e_i],e_i]),Se_j\rangle \\
    &=\sum_{1\le i,j\le n}-\frac1\alpha B([[e_j,e_i],Se_i],Se_j)+\frac1\alpha B([[e_j,e_i],e_i],e_j) \\
    &=\sum_{1\le i,j\le n} \frac{1}{\alpha}B([[Se_i, e_j], e_i], Se_j) + \frac{1}{\alpha}B([[e_j,e_i],e_i],e_j)\\
    &=\sum_{1\le i,j\le n}\frac1\alpha B([Se_i,e_j],[e_i,Se_j])-\frac1\alpha B([e_i,e_j],[e_i,e_j])\\
    &=-\frac{4}{\alpha}\sum_{n-r+1\le i\le n}\sum_{1\le j\le n-r}B([e_i,e_j],[e_i,e_j])\\
    &=\frac{4}{\alpha}\sum_{n-r+1\le i\le n}\sum_{1\le j\le n-r}\|[e_i,e_j]\|_\mathfrak{k}^2
\end{align*}

Since the Killing form $B$ is negative definite on $\mathfrak{k}$, we obtain that $H(o)\ge 0$ for $\forall o\in M$.

On the other hand, the same contraction of \eqref{centre} with \(SX\), followed by taking the trace over \(X=e_i\), gives
\[
H=2n\delta\theta-2(n-2r)\delta(S\theta)-4(r-1)r\|\theta_1\|^2-4(n-r)(n-r-1)\|\theta_2\|^2.
\]
So by integration we obtain:
\begin{equation}\label{int result 2}
    (r-1)r\int_M\|\theta_1\|^2\text{dvol}_g+(n-r)(n-r-1)\int_M\|\theta_2\|^2\text{dvol}_g=-\frac14\int_M H\text{dvol}_g.
\end{equation}
The left-hand side is non-negative and the right-hand side is non-positive, hence both sides must be zero. Since \(H\ge0\) pointwise, we also get \(H\equiv0\). The proof is decomposed into two cases:

\textbf{i):} If $r\ge 2$, then $\|\theta_1\|=\|\theta_2\|=0$, so $\theta=0$ on $M$. Thus $D$ is the Levi-Civita connection of $g$, and the lifted splitting of $\text{T}M$ on the universal cover is parallel. By the de Rham decomposition theorem, \(\widetilde M\) splits as a non-trivial Riemannian product, contradicting the irreducibility assumption.

\textbf{ii):} If \(r=1\), let \(e_1,\dots,e_{n-1}\) span \(T_1\) and let \(e_n\) span \(T_2\). Since \(H\equiv0\) on \(M\), we have $\|[e_i,e_n]\|_\mathfrak{k}=0$ for $i=1,2,\dots,n-1$. So $e_n\in Z_\mathfrak{p}(\mathfrak{p}):=\{v\in\mathfrak{p}|\forall w\in\mathfrak{p},[v,w]=0\}$, in particular $Z_\mathfrak{p}(\mathfrak{p})\neq 0$. Moreover, \(Z_{\mathfrak p}(\mathfrak p)\) is invariant under \(\operatorname{ad}(\mathfrak k)\). Indeed, if \(z\in Z_{\mathfrak p}(\mathfrak p)\), \(A\in\mathfrak k\), and \(w\in\mathfrak p\), then the Jacobi identity gives
\[
[[A,z],w]=[A,[z,w]]-[z,[A,w]]=0,
\]
because \([z,w]=0\) and \([A,w]\in\mathfrak p\). Hence \([A,z]\in Z_{\mathfrak p}(\mathfrak p)\). Since \(Z_{\mathfrak p}(\mathfrak p)\) commutes with \(\mathfrak p\) by definition, it is a non-zero abelian ideal of \(\mathfrak g\). This contradicts the semisimplicity of the transvection algebra of a symmetric space of non-compact type.
\qed

\printbibliography[title={References}]

\end{document}